\documentclass{amsart}
\usepackage[utf8]{inputenc}
\usepackage[utf8]{inputenc}
\usepackage[T1]{fontenc}
\usepackage[all]{xy}
\usepackage{lipsum}
\usepackage{url}
\usepackage{stackrel}

\usepackage{amsmath,amsthm,amssymb}

\usepackage{graphicx}

\newtheorem{thm}{Theorem}[section]

\newtheorem{prop}[thm]{Proposition}
\theoremstyle{definition}

\newtheorem{rem}[thm]{Remark}
\newtheorem{cor}[thm]{Corollary}
\numberwithin{equation}{section}

\begin{document}
\vspace{0.5in}

\renewcommand{\bf}{\bfseries}
\renewcommand{\sc}{\scshape}
\vspace{0.5in}

\title[COLLISION-FREE MOTION PLANNING ON MANIFOLDS WITH BOUNDARY]%
{COLLISION-FREE MOTION PLANNING ON MANIFOLDS WITH BOUNDARY \\}

\author{Cesar A. Ipanaque Zapata }
\date{} 
\address[Cesar A. Ipanaque. Zapata]{Deparatmento de Matem\'{a}tica,UNIVERSIDADE DE S\~{A}O PAULO
INSTITUTO DE CI\^{E}NCIAS MATEM\'{A}TICAS E DE COMPUTA\c{C}\~{A}O -
USP , Avenida Trabalhador S\~{a}o-carlense, 400 - Centro CEP:
13566-590 - S\~{a}o Carlos - SP, Brasil}
\email{cesarzapata@usp.br}



\thanks {The author wishes to acknowledge support for this research, from FAPESP 2016/18714-8.}

\begin{abstract}

This paper concerns the study of the homotopy type of the ordered configuration space for manifolds with boundary and as an application we will study the collision free motion planning problem on manifolds with boundary.
   
\end{abstract}

\maketitle

\section{\bf Introduction}
Consider the following problem in robotics. Assume, for instance, that we have $k$ objects (robots) moving in $\mathbb{D}^n:=\{x\in \mathbb{R}^n\mid ~\|x\|\leq 1\}$ with no collisions and avoiding the obstacles whose geometry is prescribed in advance. The motion planning problem consists in constructing an algorithm, which produces continuous paths in $\mathbb{D}^n$ to transport the objects from an initial configuration to final configuration such that in the process of motion there occur no collision between the objects and such that the objects do not touch the obstacles in the process of motion.

Michael Farber \cite{farber2003topological} gave the topological approach to the robot motion planning problem when the configuration space of the system is known in advance. He divided the whole configuration space of the system into pieces (local domains) and prescribed continuous motion over each of the local domains. Thus, Michael Farber defined a numerical invariant $TC(X)$, the topological complexity of the space $X$, as the minimal number of such local domains. Formally,  given a path-connected topological space $X$. The \textit{Topological complexity} of the space $X$ \cite{farber2003topological}, denoted $TC(X)$, is the least integer $m$ such that the Cartesian product $X\times X$ can be covered with $m$ open subsets $U_i$, \begin{equation*}
        X \times X = U_1 \cup U_2 \cup\cdots \cup U_m 
    \end{equation*} such that for any $i = 1, 2, \ldots , m$ there exists a continuous function $s_i : U_i \longrightarrow PX$, $\pi\circ s_i = id$ over $U_i$. If no such $m$ exists we will set $TC(X)=\infty$. Where $PX$ denote the space of all continuous paths $\gamma: [0,1] \longrightarrow X$ in $X$ and  $\pi: PX \longrightarrow X \times X$ denotes the
map associating to any path $\gamma\in PX$ the pair of its initial and end points $\pi(\gamma)=(\gamma(0),\gamma(1))$. Equip the path space $PX$ with the compact-open topology. 

One of the basic properties of $TC(X)$ is its homotopy invariance (\cite{farber2003topological}, Theorem 3).

The ordered configuration space of $k$ distinct points of a topological space $X$ (see \cite{fadell1962configuration}) is the subset 
 \[F(X,k)=\{(x_1,\ldots,x_k)\in X^k\mid ~~x_i\neq x_j\text{ for all } i\neq j \}\] topologised, as a subspace of the cartesian power $X^k$. Clearly, this space can be used in robotics when one controls multiple objects simultaneously, trying to avoid collisions between them \cite{farber2008invitation}. Here $k\geq 1$ is an integer.

In this paper, we make several assumptions: (a) each object is represented by a single point; (b) each obstacle is represented by small balls, possibly of different radii; (c) the obstacles are known in advance; (d) collision between two objects occurs if they are situated at the same point in space; (e) an object touches an obstacle if the point representing the object is situated at the obstacle.

One of our results (Theorem \ref{theor}) shows that the problem becomes topologically equivalent to the similar problem when the objects are restricted to do not touch the boundary in the process of motion.


\section{\bf Configuration spaces of manifolds with boundary}

Throughout this section, we will assume $M$ is a connected and compact topological manifold with nonempty boundary \cite{lee2012introduction}. Furthermore $Int(M)$ denotes the interior of the manifold $M$, that is, $Int(M)=M- \partial M$.

\subsection{Homotopy type}
 
 One principal Theorem which will be proved in this paper is as follow.
 
 \begin{thm}\label{theor} 
 For each $k\geq 1$, the inclusion map $i:Int(M)\hookrightarrow M$ induces homotopy equivalences in the configuration space $F(M,k)$, that is, the map \begin{equation}
 \begin{array}{rccl}
 F(i,k):&F(Int(M),k)&\longrightarrow & F(M,k)\\&(x_1,\ldots,x_k)&\mapsto
 &(x_1,\ldots,x_k)
                           \end{array}
 \end{equation} is a homotopy equivalence. Furthermore, the configuration space $F(M,k)$ is $\Sigma_k-$equivariantly homotopy equivalent to the configuration space $F(Int(M),k)$.
\end{thm}
\begin{proof}
Let $V$ be the open collar of $\partial M$ in $M$, where $h:\partial M\times (-1,1]\cong V$ and $\partial M$ is identified with $\partial M\times \{1\}$.

We consider the closed set in $M$,\begin{equation}\label{F}
F=M-h(\partial M\times (0,1])\subset Int(M)
\end{equation} and define the map $p:M\longrightarrow Int(M)$ given by: \[ \forall
y\in M,~~
 p(y):=\left\{
              \begin{array}{ll}
                h(x,\frac{1}{2}t), & \hbox{if $y=h(x,t),~~0\leq t\leq 1$.} \\
                y, & \hbox{if $y\in F$.}
              \end{array}
            \right.
  \] Note that the map $p$ is continuous and injective.
  
  To continue, define the isotopy\footnote{that is, $H_t$ is injective, $\forall t\in [0,1]$} $H:M\times [0,1]\longrightarrow M$ given by: \[ \forall
y\in M, \;\; \forall \;\; s\in [0,1],~~
 H(y,s):=\left\{
              \begin{array}{ll}
                h(x,(\frac{1}{2}+\frac{s}{2})t), & \hbox{if $y=h(x,t),~~0\leq t\leq 1$,} \\
                y, & \hbox{if $y\in F$;}
              \end{array}
            \right.
  \] so that $H_0=i\circ p$ and $H_1=id_{M}$. Furthermore, note that by restricting $H$ to $Int(M)\times [0,1]$, we obtain an isotopy $G$ between $p\circ i$ and $id_{Int(M)}$.
  
  Moreover the isotopies $H$ and $G$ induce homotopy equivalences in the configuration spaces,
  \[ \begin{array}{rccl}
 F(H,k):&F(M,k)\times [0,1]&\longrightarrow & F(M,k)\\&(x_1,\ldots,x_k,t)&\mapsto
 &(H(x_1,t),\ldots,H(x_k,t))
                           \end{array} \]
 and \[ \begin{array}{rccl}
 F(G,k):&F(Int(M),k)\times [0,1]&\longrightarrow & F(Int(M),k)\\&(x_1,\ldots,x_k,t)&\mapsto
 &(G(x_1,t),\ldots,G(x_k,t))
                           \end{array} \] so that $F(H,k):F(i,k)\circ F(p,k)\simeq id$ and $F(G,k):F(p,k)\circ F(i,k)\simeq id$, where  \[ \begin{array}{rccl}
 F(p,k):&F(M,k)&\longrightarrow & F(Int(M),k)\\&(x_1,\ldots,x_k)&\mapsto
 &(p(x_1),\ldots,p(x_k))
                           \end{array} \]
 and \[ \begin{array}{rccl}
 F(i,k):&F(Int(M),k)&\longrightarrow & F(M,k)\\&(x_1,\ldots,x_k)&\mapsto
 &(x_1,\ldots,x_k)
                           \end{array} \] This completes the proof of the theorem.
  
\end{proof}

\begin{rem}\label{rem-puctured}
The assertion of Theorem \ref{theor} can be strengthened for punctured manifolds, that is,  for each $k\geq 1$, the inclusion map $i:Int(M)-Q_m\hookrightarrow M-Q_m$ induces homotopy equivalences between the configuration spaces $F(Int(M)-Q_m,k)$ and $F(M-Q_m,k)$, where $Q_m\subset F\subset Int(M)$ is a fixed sequence of $m$ distinct points in $F\subset Int(M)$ (see (\ref{F})).
\end{rem}
 
\begin{cor}\label{coro}
For each positive integer $k$ the space $F(\mathbb{D}^n,k)$ is a deformation retract of the space $F(\mathbb{R}^n,k)$. Moreover, there is a $\Sigma_k-$equivariant deformation retraction of $F(\mathbb{R}^n,k)$ onto $F(\mathbb{D}^n,k)$.
\end{cor}

Theorem \ref{theor-F-G} we state in this section is known, It can be found in the paper by Michael Farber and Mark Grant \cite{farber2009topological}. 

\begin{thm}(\cite{farber2009topological}, Theorem 1.1)\label{theor-F-G}
For $n,k\geq 2$,
\begin{equation}
  TC(F(\mathbb{R}^n,k))=\left\{ \begin{array}{ll}
                                                           2k-1, & \hbox{ for $n$ odd;} \\
                                                           2k-2, & \hbox{ for $n$ even.}
                                                         \end{array}
                                                       \right.
  \end{equation}
\end{thm}

As an immediately consequence we have: 
 
\begin{thm}\label{tc-dn}
  For each positive integer $n\geq 2$ and $k\geq 2$,
  \begin{equation}
  TC(F(\mathbb{D}^n,k))=\left\{ \begin{array}{ll}
                                                           2k-1, & \hbox{ for $n$ odd;} \\
                                                           2k-2, & \hbox{ for $n$ even.}
                                                         \end{array}
                                                       \right.
  \end{equation}
\end{thm}
\begin{proof}
We apply Corollary \ref{coro} which together with homotopy invariance of TC gives the proof.
\end{proof}

\begin{rem}
For $k>r\geq 1$ there is a natural map
\begin{equation}
\pi_{k,r}:F(M,k)\longrightarrow F(M,r)
\end{equation} obtained by projecting to the first $r$ factors. 

It is well known that the restriction map \begin{equation}
\pi_{k,r}:F(Int(M),k)\longrightarrow F(Int(M),r)
\end{equation} is a fibration with fiber $F(Int(M)-Q_r,k-r)$, when $dim(M)\geq 2$. In contrast, the map $\pi_{k,r}:F(M,k)\longrightarrow F(M,r)$ is not a fibration. The fact that the map $\pi_{k,r}:F(M,k)\longrightarrow F(M,r)$ is not a fibration may be seen by considering, for example, the manifold $\mathbb{D}^2$ that is with boundary but the fiber $\mathbb{D}^2-\{(0,0)\}$ is not homotopy equivalent to the fiber $\mathbb{D}^2-\{(1,0)\}$.

It is easy to check that there is a commutative diagram 
\begin{equation}
\xymatrix{ F(M-Q_r,k-r) \ar[d]_{F(p,k-r)} \ar[r]^{\quad\quad j} \ar@{}[dr]|{\circlearrowleft} & 
F(M,k) \ar[d]_{F(p,k)} \ar[r]^{\pi_{k,r}} \ar@{}[dr]|{\circlearrowleft} & F(M,r) \ar[d]^{F(p,r)} \\ F(Int(M)-Q_r,k-r) \ar[r]_{\quad\quad j} & F(Int(M),k) \ar[r]_{\pi_{k,r}} & F(Int(M),r)
}
\end{equation}
where $j$ is the inclusion map, $F(p,k)$ was defined by Theorem \ref{theor} and $Q_m\subset F\subset Int(M)$ is a fixed sequence of $m$ distinct points in $Int(M)$ (see Remark \ref{rem-puctured}). 

In this case the long exact sequence of homotopy groups of the fibration  \[\pi_{k,r}:F(Int(M),k)\longrightarrow F(Int(M),r)\] can induce a long exact sequence of homotopy groups for the map  $\pi_{k,r}:F(M,k)\longrightarrow F(M,r)$ (see Diagram (\ref{induz-seq})). Where \begin{equation} \partial:=F(i,k-r)_\ast \circ \delta\circ F(p,k)_\ast, \end{equation} we recall $(F(p,k-r)_\ast)^{-1}=F(i,k-r)_\ast$.
\end{rem}

\newpage
\rotatebox{90}{
\begin{minipage}{\textheight}
\vbox{
\begin{equation}\label{induz-seq}
 \xymatrix{  \ar[r] & \pi_{q+1}(F(M,r)) \ar[r]^{\partial\quad\quad} \ar[d]_{F(p,r)_\ast} \ar@{}[dr]|{\circlearrowleft}& \pi_q(F(M-Q_r,k-r)) \ar[r]^{\quad\quad j_\ast} \ar[d]_{F(p,k-r)_\ast}  \ar@{}[dr]|{\circlearrowleft} & \pi_q(F(M,k)) \ar[r]^{(\pi_{k,r})_\ast} \ar[d]_{F(p,k)_\ast} \ar@{}[dr]|{\circlearrowleft} & \pi_q(F(M,r))  \ar[d]_{F(p,r)_\ast}  \\ \ar[r] & \pi_{q+1}(F(Int(M),r)) \ar[r]_{\delta\quad\quad} & \pi_q(F(Int(M)-Q_r,k-r)) \ar[r]_{\quad\quad j_\ast} & \pi_q(F(Int(M),k)) \ar[r]_{(\pi_{k,r})_\ast} & \pi_q(F(Int(M),r))    
 }
\end{equation}
}
\end{minipage}

}
\newpage

\subsection{On the homotopy fiber of the inclusion map $F(M,k)\hookrightarrow \prod_{1}^{k}M$}

Recently, the homotopy fiber of the inclusion map $i_k(X):F(X,k)\hookrightarrow \prod_{1}^{k}$ was considered by Golasi{\'n}ski,  Gon{\c{c}}alves and Guaschi in \cite{golasinski2017homotopy} when $X$ is a topological manifold without boundary.

Let $X$ and $Y$ topological spaces. Given a map $f:X\longrightarrow Y$ and $y_0\in Y$, we will consider
\begin{equation}
E_f=\{(x,\gamma)\in X\times PY\mid~\gamma(0)=f(x)\} \text{ and } I_f=\{(x,\gamma)\in E_f\mid~\gamma(1)=y_0\}
\end{equation} the \textit{mapping path} and \textit{homotopy fibre} of $f$, respectively.

It is well known that \begin{equation}
p:E_f\longrightarrow Y,~(x,\gamma)\mapsto p(x,\gamma)=\gamma(1),
\end{equation} is a fibration with fiber $I_f$ and that $E_f$ and $X$ have the same homtopy type (\cite{hatcher2002algebraic}, Proposition 4.64).

The following theorem, which is the main result of this section, describes the homotopy type of the homotopy fiber of the inclusion  map $i_k(M):F(M,k)\hookrightarrow \prod_{1}^{k}M$.

\begin{thm}\label{homotopy-fiber}
\begin{equation}
I_{i_k(M)}\simeq I_{i_k(Int(M))}.
\end{equation} Where we write $X\simeq Y$ if $X$ and $Y$ have the same homotopy type.
\end{thm}
\begin{proof}
Since $F(p,k):F(M,k)\longrightarrow F(Int(M),k)$ and $\prod_{1}^{k}p:\prod_{1}^{k}M\longrightarrow \prod_{1}^{k}Int(M)$ are homotopy equivalences (see Theorem \ref{theor}), then by applying (\cite{mather1976pull}, Lemma 45) to the following homotopy commutative diagram:
\begin{equation}
\xymatrix{ F(M,k) \ar[d]_{F(p,k)} \ar[r]^{\quad i_k(M)} \ar@{}[dr]|{\circlearrowleft} & 
\prod_{1}^{k}M \ar[d]^{\prod_{1}^{k}p} \\ F(Int(M),k)\quad \ar[r]_{\quad i_k(Int(M))\quad} &\quad \prod_{1}^{k}Int(M)
}
\end{equation}
We see that $I_{i_k(M)}$ has the same homotopy type of  $I_{i_k(Int(M))}$.
\end{proof}

We recall that a map $f:X\longrightarrow Y$ of pointed spaces is said to be \textit{$n$-connected} if the induced homomorphism $f_\ast:=\pi_q(f):\pi_q(X)\longrightarrow \pi_q(Y)$ is surjective if $q=n$ and is an isomorphism for all $0\leq q\leq n-1$. It is well known that the map $f$ is $n-$connected if and only if its homotopy fiber $I_f$ is $(n-1)-$connected, that is, $\pi_q(I_f)=0$ for all $0\leq q\leq n-1$.

Theorem \ref{marek} we state in this section is known, It can be found in the paper by Golasi{\'n}ski,  Gon{\c{c}}alves and Guaschi \cite{golasinski2017homotopy}. 

\begin{thm}\label{marek}(\cite{golasinski2017homotopy}, Theorem 3.2)
Let $X$ be a connected, topological manifold without boundary. Then the inclusion map $i_k(X):F(X,k)\hookrightarrow \prod_{1}^{k} X$ is $(dim(X)-1)-$connected.
\end{thm}

As a consequence of Theorem \ref{homotopy-fiber} and Theorem \ref{marek}, we have the following result.

\begin{prop}\label{dim-con} Let $M$ be a connected, compact topological manifold with boundary. Then the inclusion map $i_k(M):F(M,k)\hookrightarrow \prod_{1}^{k} M$ is $(dim(M)-1)-$connected.
\end{prop}

\begin{rem}
The assertion of Theorem \ref{homotopy-fiber} and Proposition \ref{dim-con} can be strengthened for punctured manifolds, that is,  for each $k\geq 1$, the homotopy fiber $I_{i_k(M-Q_m)}$ has the same homotopy type of  $I_{i_k(Int(M)-Q_m)}$. Furthermore, the inclusion map $i_k(M-Q_m):F(M-Q_m,k)\hookrightarrow \prod_{1}^{k} (M-Q_m)$ is $(dim(M)-1)-$connected. Where $Q_m\subset F\subset Int(M)$ is a fixed sequence of $m$ distinct points in $F\subset Int(M)$ (see (\ref{F})).
\end{rem}


\section{\bf Sphere world}

A $n-$dimensional \textit{sphere world} \cite{koditschek1990robot} is a compact connected subset of  $\mathbb{R}^n$ 
whose boundary is formed from disjoint union of a finite number of $(n-1)-$spheres. 

In this paper we will consider the following classical sphere world \cite{koditschek1990robot}. A large sphere which bounds the \textit{work-space}, \begin{equation}
W:=\{x\in \mathbb{R}^n\mid~\| x\|\leq r_0\}
\end{equation} and $m$ smaller $(n-1)-$spheres which bound the \textit{obstacles},
\begin{equation}
O_i:=\{x\in\mathbb{R}^n\mid~\|x-x_i\|< r_i\},~i=1,\ldots,m
\end{equation}

The \textit{free-space} remains after removing all the obstacles from the work-space,
\begin{equation}
X_{n,m}:=W\setminus \bigcup_{i=1}^{m}O_i
\end{equation}

Here $n\geq 2$ and $m$ are integer numbers. For $X_{n,m}$ to be a sphere world we must impose that all obstacle closures are contained in the interior of the work-space (that is, $r_0>0$ and $\|x_i\|+r_i<r_0,~i=1,\ldots,m$) and that none of them intersect (that is, $\|x_1-x_j\|>r_i+r_j,~i,j=1,\ldots,m$). For ease of exposition we center the work-space at the origin of $\mathbb{R}^n$.

\begin{figure}[!h]
\caption{$n-$dimensional sphere world $X_{n,m}$.}
 \label{n-dimesional-sphere-world}
 \centering
 \includegraphics[scale=0.8]{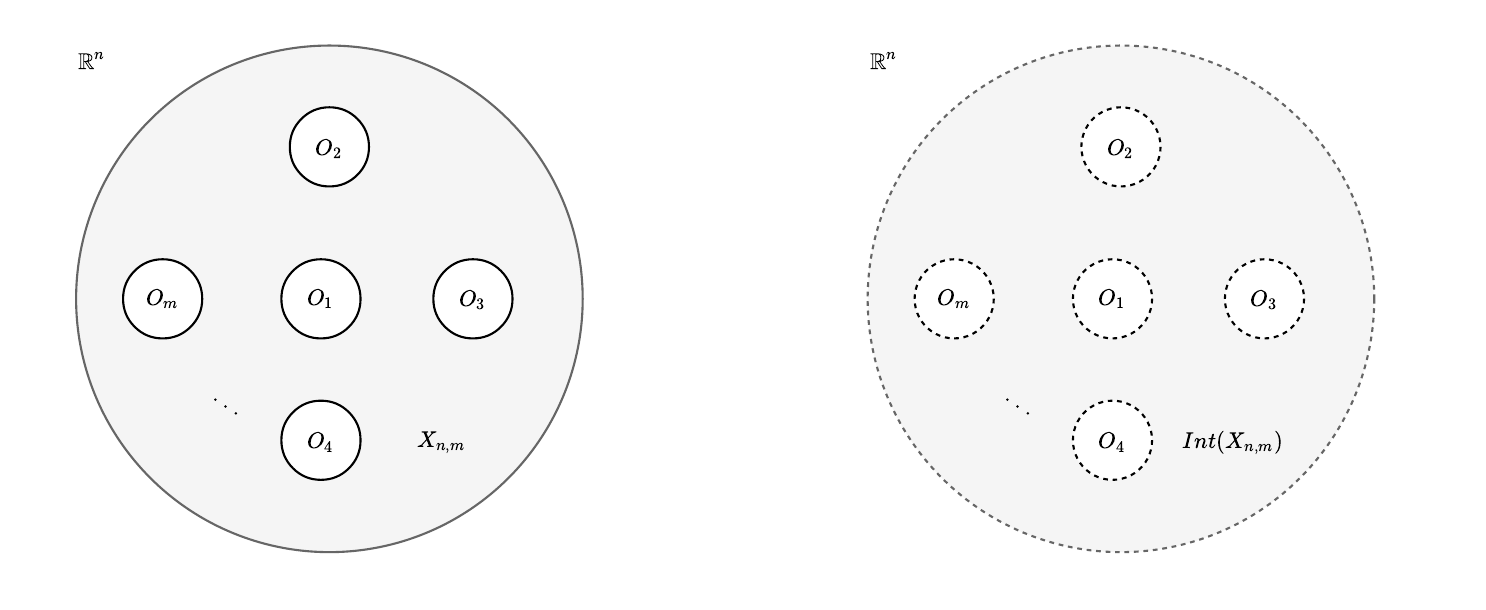}
\end{figure}

Clearly, $X_{n,m}$ is a $n-$dimensional connected and compact manifold with nonempty boundary. Let $Q_m=\{q_1,\ldots,q_m\}\subset \mathbb{R}^n$ be a fixed sequence of $m$ distinct points in the Euclidean space $\mathbb{R}^n$. 

We recall the following proposition.
\begin{prop}\label{p-1}
\begin{equation}Int(X_{n,m})\cong \mathbb{R}^n-Q_m.\end{equation}
\end{prop}

Now, we recall Theorem \ref{F-G-Y},  It can be found in the paper by Michael Farber, Mark Grant and Sergey Yuzvinsky \cite{farber2007topological}. We have adjusted some notations as required for our needs. 

\begin{thm}(\cite{farber2007topological}, Theorem 6.1 and Theorem 5.1)\label{F-G-Y} One has
\begin{equation} TC(F(\mathbb{R}^2-Q_m,k))=\left\{
              \begin{array}{ll}
                2k-2, & \hbox{if $m=0$.} \\
                2k, & \hbox{if $m=1$.}\\
                2k+1, & \hbox{if $m\geq 2$.}
              \end{array}
            \right.\end{equation}
 and \begin{equation} TC(F(\mathbb{R}^3-Q_m,k))=\left\{
              \begin{array}{ll}
                2k-1, & \hbox{if $m=0$.} \\
                2k+1, & \hbox{if $m\geq 1$.}
              \end{array}
            \right.\end{equation}
\end{thm}

\begin{rem}
Theorem \ref{F-G-Y} is independent of the dimension of $\mathbb{R}^n$ (\cite{gonzalez2015sequential}, Theorem 1.3, pg. 4505), that is, 
\begin{equation}
TC(F(\mathbb{R}^n-Q_m,k))=\left\{
              \begin{array}{ll}
               TC(F(\mathbb{R}^2-Q_m,k)), & \hbox{if $n\geq 2$ even.} \\
            TC(F(\mathbb{R}^3-Q_m,k)), & \hbox{if $n\geq 3$ odd.}
              \end{array}
            \right.
\end{equation} 
\end{rem}

\begin{thm}\label{theor-principal} If $n\geq 2$ even, then
\begin{equation} TC(F(X_{n,m},k))=\left\{
              \begin{array}{ll}
                2k-2, & \hbox{if $m=0$.} \\
                2k, & \hbox{if $m=1$.}\\
                2k+1, & \hbox{if $m\geq 2$.}
              \end{array}
            \right.\end{equation}
 If $n\geq 3$ odd, then \begin{equation} TC(F(X_{n,m},k))=\left\{
              \begin{array}{ll}
                2k-1, & \hbox{if $m=0$.} \\
                2k+1, & \hbox{if $m\geq 1$.}
              \end{array}
            \right.\end{equation}
\end{thm}
\begin{proof}
Note that this theorem follows by simply combining Theorems \ref{theor}, \ref{F-G-Y} and Proposition \ref{p-1}.
\end{proof}

\begin{rem}
Theorem \ref{theor-principal} reflects  that the conclusions of the paper \cite{farber2007topological} remain valid in a more general and realistic situations when the obstacles are represented by balls and it is independent of the collisions with the boundary.
\end{rem}

\begin{rem}
As an application of Configuration spaces of hard spheres \cite{yuliy1014mintype},  we believe that Theorem \ref{theor-principal} will remain valid when the objects are represented by small balls and the control requirements are to void tangencies between objects and obstacles.
\end{rem}

\bibliographystyle{plain}

\end{document}